\documentclass[12pt]{amsart}

\newtheorem{theorem}{Theorem}[section]
\newtheorem{lemma}[theorem]{Lemma}

\theoremstyle{plain}
\newtheorem{definition}[theorem]{Definition}

\theoremstyle{definition}

\theoremstyle{remark}

\numberwithin{equation}{section}

\usepackage{fancyhdr}
\usepackage{amscd}
\usepackage{amsmath}
\usepackage{amsthm}
\usepackage{amssymb}

\begin{document}

\title[Zero map between obstruction spaces]{Zero map between obstruction spaces: subvarieties versus cycles}

\author{Sen Yang}
\address{Yau Mathematical Sciences Center, Tsinghua University
\\
Beijing, China}
\email{syang@math.tsinghua.edu.cn; senyangmath@gmail.com}

\subjclass[2010]{14C25}
\date{}

\begin{abstract}
 For $Y \subset X$ a locally complete intersection of codimension $p$, Spencer Bloch \cite{Bloch2} constructed the semi-regularity map $\pi:  H^{1}(\mathcal{N}_{Y/X}) \to H^{p+1}(\Omega_{X/k}^{p-1})$.
As an analogue, we construct a map $\tilde{\pi}:  H^{1}(\mathcal{N}_{Y/X}) \to H^{p+1}(\Omega_{X/\mathbb{Q}}^{p-1})$, 
without assuming locally complete intersections. 

While the semi-regularity map $\pi$ is expected to be injective in \cite{Bloch2}, we show that $\tilde{\pi}$ is a zero map. 
We use this zero map to interpret how to eliminate obstructions to deforming cycles, an idea by Mark Green and Phillip Griffiths in \cite{GGtangentspace}.
 
\end{abstract}

\maketitle

\tableofcontents

\section{\textbf{Introduction}}
\label{Introduction}





Let $X$ be a nonsingular projective variety over a field $k$ of characteristic $0$. For $Y \subset X$ a subvariety of codimension $p$, it is well-known that the lifting of $Y$ to higher order may be obstructed. 

However, Green-Griffiths predicts that we can eliminate obstructions in their program \cite{GGtangentspace}, by considering $Y$ as a \textbf{cycle}. 
For $p=1$, Green-Griffiths' idea was realized by TingFai Ng in his Ph.D thesis \cite{Ng}. 
For general $p$($1\leqslant p \leqslant \mathrm{dim}(X)$), Mark Green and Phillip Griffiths asked an open question on eliminating obstructions in \cite{GGtangentspace}(page 187-190),  which has been reformulated and has been answered affirmatively in \cite{Y-33}. In \cite{Ng}, TingFai Ng asked a concrete question of eliminating obstructions to deforming curves on a three-fold, which has been studied in \cite{Y-4}.

On the other hand, for $Y$ a locally complete intersection of codimension $p$, Spencer Bloch \cite{Bloch2} constructed the semi-regularity map
 \[
 \pi:  H^{1}(\mathcal{N}_{Y/X}) \to H^{p+1}(\Omega_{X/k}^{p-1}),
 \]
 generalizing the construction by Kodaira-Spencer \cite{KS} for the divisor case.
Guided by Spencer Bloch's semi-regularity map and Mark Green and Phillip Griffiths' idea on eliminating obstructions to deforming cycles, we construct a map 
 \[
 \tilde{\pi}:  H^{1}(\mathcal{N}_{Y/X}) \to H^{p+1}(\Omega_{X/\mathbb{Q}}^{p-1}).
 \]
In fact, the map $\tilde{\pi}$ can be immediately deduced from the sheafification of the map defined in Definition 4.1 in \cite{Y-3}. For the readers' convenience, we sketch it very briefly as follows.

For $X$ a nonsingular projective variety over a field $k$ of characteristic $0$, let $Y \subset X$ be a subvariety of codimension $p$, the normal sheaf $\mathcal{N}_{Y/X}$ is defined to be $\it {Hom}_{Y}(I/I^{\mathrm{2}}, O_{Y})$, where $I$ is the ideal sheaf of $Y$ in $X$. For any arbitrary open $U \subset X$ with $U \cap Y \neq \emptyset$, 
\[
\Gamma(U, \mathcal{N}_{Y/X}) = \mathrm{Hom}_{U \cap Y}(I/I^{2}\mid_{U \cap Y}, O_{U \cap Y}).
\]

To fix notations, $k[\varepsilon]/(\varepsilon^{2})$ is the ring of dual numbers. Let $(U \cap Y)^{'}$ be a first order infinitesimal deformation of $U \cap Y$ in $U$, that is, $(U \cap Y)^{'} \subset U[\varepsilon]/(\varepsilon^{2})$ such that $(U \cap Y)^{'}$ is flat over $\mathrm{Spec}(k[\varepsilon]/(\varepsilon^{2}))$ and $(U \cap Y)^{'} \otimes_{\mathrm{Spec}(k[\varepsilon]/(\varepsilon^{2}))} \mathrm{Spec}(k) \cong U \cap Y$. 


Let $y$ be the generic point of $Y$, $U \cap Y$ is generically generated by a regular sequence of length $p$: $f_{1}, \cdots, f_{p}$, and $(U \cap Y)^{'}$ is generically given by $f_{1}+\varepsilon g_{1}, \cdots, f_{p}+\varepsilon g_{p}$, where $g_{1}, \cdots, g_{p} \in O_{U, y}$.

We use $F_{\bullet}(f_{1}+\varepsilon g_{1}, \cdots, f_{p}+\varepsilon g_{p})$ to denote the Koszul complex associated to the  regular sequence $f_{1}+\varepsilon g_{1}, \cdots, f_{p}+\varepsilon g_{p}$:
\[
 \begin{CD}
  0 @>>> F_{p} @>A_{p}>> F_{p-1} @>A_{p-1}>>  \dots @>A_{2}>> F_{1} @>A_{1}>> F_{0},
 \end{CD}
\]
where each $F_{i}=\bigwedge^{i} O^{\bigoplus p}_{U,y}[\varepsilon]/(\varepsilon^{2})$ and $A_{i}: \bigwedge^{i} O^{\bigoplus p}_{U,y}[\varepsilon]/(\varepsilon^{2})  \to \bigwedge^{i-1} O^{\bigoplus p}_{U,y}[\varepsilon]/(\varepsilon^{2})$ are defined as usual.

Recall that Milnor K-groups with support are rationally defined in terms of eigenspaces of Adams operations in \cite{Y-2}:
\begin{definition}  [Definition 3.2 in \cite{Y-2}] \label{definition:Milnor K-theory with support}
Let $X$ be a finite equi-dimensional noetherian scheme and $x \in X^{(j)}$. For $m \in \mathbb{Z}$, Milnor K-group with support $K_{m}^{M}(O_{X,x} \ \mathrm{on} \ x)$ is rationally defined to be 
\[
  K_{m}^{M}(O_{X,x} \ \mathrm{on} \ x) := K_{m}^{(m+j)}(O_{X,x} \ \mathrm{on} \ x)_{\mathbb{Q}},
\] 
where $K_{m}^{(m+j)}$ is the eigenspace of $\psi^{k}=k^{m+j}$ and $\psi^{k}$ is the Adams operations.

\end{definition}

Here, $K_{m}(O_{X,x} \ \mathrm{on} \ x)$ is Thomason-Trobaugh K-group and $K_{m}^{(m+j)}(O_{X,x} \ \mathrm{on} \ x)_{\mathbb{Q}}$ denotes the image of $K_{m}^{(m+j)}(O_{X,x} \ \mathrm{on} \ x)$ in $K_{m}^{(m+j)}(O_{X,x} \ \mathrm{on} \ x) \otimes_{\mathbb{Z}} \mathbb{Q}$.

\begin{theorem} [Prop 4.12  of \cite{GilletSoule}]  \label{theorem: GilletSoule}
The Adams operations $\psi^{k}$ defined on perfect complexes, defined by Gillet-Soul\'e in \cite{GilletSoule}, satisfy
$\psi^{k}(F_{\bullet}(f_{1}+\varepsilon g_{1}, \cdots, f_{p}+\varepsilon g_{p})) = k^{p}F_{\bullet}(f_{1}+\varepsilon g_{1}, \cdots, f_{p}+\varepsilon g_{p})$.
\end{theorem}

Hence, $F_{\bullet}(f_{1}+\varepsilon g_{1}, \cdots, f_{p}+\varepsilon g_{p})$ is of eigenweight $p$ and can be considered as an element of $K^{(p)}_{0}(O_{U,y}[\varepsilon] \ \mathrm{on} \ y[\varepsilon])_{\mathbb{Q}}$: 
\[
F_{\bullet}(f_{1}+\varepsilon g_{1}, \cdots, f_{p}+\varepsilon g_{p}) \in K^{(p)}_{0}(O_{U,y}[\varepsilon] \ \mathrm{on} \ y[\varepsilon])_{\mathbb{Q}}=K^{M}_{0}(O_{U,y}[\varepsilon] \ \mathrm{on} \ y[\varepsilon]).
\]

It is known that $(U \cap Y)^{'}$ corresponds to an element of $\Gamma(U, \mathcal{N}_{Y/X})$:
\begin{lemma} [Theorem 2.4 in \cite{Hartshorne2}]
 The first order infinitesimal deformations of  $U \cap Y$ in $U$ are one-to-one correspondence with elements of 
 $\Gamma(U, \mathcal{N}_{Y/X})(= \Gamma(U \cap Y, \mathcal{N}_{U \cap Y/ U}))$.
\end{lemma}
So we have the following map(Definition 2.4 in \cite{Y-3}),
\begin{equation}
\mu:   \ \Gamma(U, \   \mathcal{N}_{Y/X}) \longrightarrow K^{M}_{0}(O_{U,y}[\varepsilon] \ \mathrm{on} \ y[\varepsilon])
\end{equation}
 \[
 (U \cap Y)' \longrightarrow  F_{\bullet}(f_{1}+\varepsilon g_{1}, \cdots, f_{p}+\varepsilon g_{p}).
\]

Let $K^{M}_{0}(O_{U,y}[\varepsilon] \ \mathrm{on} \ y[\varepsilon],\varepsilon)$ denote the relative K-group, that is, the kernel of the natural projection
\[
K^{M}_{0}(O_{U,y}[\varepsilon] \ \mathrm{on} \ y[\varepsilon]) \xrightarrow{\varepsilon =0} K^{M}_{0}(O_{U,y} \ \mathrm{on} \ y).
\] 

We have shown that relative Chern character induces the following isomorphisms between relative K-groups and local cohomology groups, see 
Corollary 3.2 in \cite{Y-3}, which is a particular case of Corollary 9.5 in \cite{DHY} or Corollary 3.11 in \cite{Y-2} \footnote{The hypothesis ``projective" therein can be loosened, $O_{U,y}$ is still a regular local ring of dimension $p$. The proofs in \cite{DHY} and \cite{Y-2} go straightly},
\begin{equation}
 K^{M}_{0}(O_{U,y}[\varepsilon] \ \mathrm{on} \ y[\varepsilon],\varepsilon) \xrightarrow{\cong} H_{y}^{p}(\Omega^{p-1}_{U/\mathbb{Q}}).
\end{equation}

Composing the natural projection
\[
K^{M}_{0}(O_{U,y}[\varepsilon] \ \mathrm{on} \ y[\varepsilon]) \to  K^{M}_{0}(O_{U,y}[\varepsilon] \ \mathrm{on} \ y[\varepsilon],\varepsilon),
\] 
with the above isomorphism(1.2), one has the following natural surjective map(Definition 3.3 in \cite{Y-3}),
\begin{equation}
\mathrm{Ch}:  K^{M}_{0}(O_{U,y}[\varepsilon] \ \mathrm{on} \ y[\varepsilon]) \to H_{y}^{p}(\Omega_{U/\mathbb{Q}}^{p-1}).
\end{equation}

This Ch map can be described by a beautiful construction of B. Ang\'eniol and M. Lejeune-Jalabert \cite{A-LJ}, 
see also page 5-6 of \cite{Y-3} for a brief summary. For our purpose, the image of the Koszul complex $F_{\bullet}(f_{1}+\varepsilon g_{1}, \cdots, f_{p}+\varepsilon g_{p})$ under the Ch map can be described explicitly.
For simplicity, we assume $g_{2}= \cdots =g_{p}=0$, then the image of
 $F_{\bullet}(f_{1}+\varepsilon g_{1}, f_{2}, \cdots, f_{p})$ under the map Ch, can be represented by the following diagram
\[
\begin{cases}
 \begin{CD}
    F_{\bullet}(f_{1}, \cdots, f_{p}) @>>> O_{U, y}/(f_{1}, \cdots,  f_{p}) \\
   F_{p}(\cong O_{U,y}) @>g_{1}df_{2} \wedge \dots \wedge df_{p}>> F_{0} \otimes \Omega_{O_{U,y}/ \mathbb{Q}}^{p-1}(\cong\Omega_{O_{U,y}/ \mathbb{Q}}^{p-1}),
 \end{CD}
\end{cases}
 \]
 where $d=d_{\mathbb{Q}}$ and $F_{\bullet}(f_{1}, \cdots, f_{p})$ is the Koszul complex associated to the  regular sequence $f_{1}, \cdots, f_{p}$.

We define the following map
\begin{equation}
\pi_{U}: \Gamma(U, \mathcal{N}_{ Y/ X}) \to  H_{y}^{p}(\Omega_{U/\mathbb{Q}}^{p-1}),
\end{equation}
by composing Ch with $\mu$:
\[
 \Gamma(U, \mathcal{N}_{ Y/ X}) \xrightarrow[\mathrm{(1.1)}]{\mu}   K^{M}_{0}(O_{U,y}[\varepsilon] \ \mathrm{on} \ y[\varepsilon]) \xrightarrow[\mathrm{(1.3)}]{\mathrm{Ch}} H_{y}^{p}(\Omega_{U/\mathbb{Q}}^{p-1}).
\]

For any inclusion $V \subset U$ such that $V \cap Y \neq \emptyset$, one checks the following diagram is commutative
\[
\begin{CD}
\Gamma(U, \mathcal{N}_{ Y/ X}) @>\pi_{U}>>  H_{y}^{p}(\Omega_{U/\mathbb{Q}}^{p-1})\\
   @VVV @VV=V \\
  \Gamma(V, \mathcal{N}_{ Y/ X}) @>\pi_{V}>> H_{y}^{p}(\Omega_{V/\mathbb{Q}}^{p-1}).
 \end{CD}
\]

\section{\textbf{Theorem and interpretation}}
Recall that $X$ is a nonsingular projective variety over a field $k$ of characteristic $0$, $Y \subset X$ is a subvariety of codimension $p$.
Let $y$ be the generic point of $Y$, one defines a presheaf $\underline{H}_{y}^{p}(\Omega_{X/\mathbb{Q}}^{p-1})$ on $X$ as follows
\[
\begin{cases}
 U \to  &  H_{y}^{p}(U, \Omega_{X/\mathbb{Q}}^{p-1}), \  \mathrm{if} \  y \in U, \\
  U \to & 0,    \ \mathrm{if} \ y \notin U.
 \end{cases}
\]
In fact, $\underline{H}_{y}^{p}(\Omega_{X/\mathbb{Q}}^{p-1})$ is a sheaf, which is identified with the sheaf $i_{y}\ast H_{y}^{p}(\Omega_{X/\mathbb{Q}}^{p-1})$, where $i_{y}: \{ y\} \hookrightarrow X$ is the immersion and $H_{y}^{p}(\Omega_{X/\mathbb{Q}}^{p-1})$ is a constant sheaf on $y$. Moreover, $i_{y}\ast H_{y}^{p}(\Omega_{X/\mathbb{Q}}^{p-1})$ is flasque, so is $\underline{H}_{y}^{p}(\Omega_{X/\mathbb{Q}}^{p-1})$.

\begin{definition}\label{definition: composition}
In the notation above, there exists the following morphism of sheaves
\begin{equation}
\pi :  \mathcal{N}_{ Y/ X} \to  \underline{H}_{y}^{p}(\Omega_{X/\mathbb{Q}}^{p-1}),
\end{equation}
 which is locally defined by (1.4).
\end{definition}

The map $\pi$(2.1) induces a map between cohomology groups:
\begin{equation}
\pi : H^{1}(\mathcal{N}_{ Y/ X}) \to H^{1}(\underline{H}_{y}^{p}(\Omega_{X/\mathbb{Q}}^{p-1})).
\end{equation}
There exists the following spectral sequence, see Motif E on page 221 of \cite{Hartshorne},
\[
 H^{m}(X, \underline{H}_{y}^{n}(\Omega_{X/\mathbb{Q}}^{p-1})) \Rightarrow H_{y}^{k}(\Omega_{X/\mathbb{Q}}^{p-1}), \ \mathrm{with} \ m+n = k.
\]

$\Omega_{X/\mathbb{Q}}^{p-1}$
can be written as a direct limit of direct sum of $O_{X}$(as $O_{X}$-module), though it is not of finite type. 
Since $O_{X}$ is Cohen-Macaulay, $\underline{H}_{y}^{n}(O_{X}) = 0 $, unless $n=p$.
Consequently, $\underline{H}_{y}^{n}(\Omega_{X/\mathbb{Q}}^{p-1}) = 0 $, unless $n=p$. For our purpose, by taking $k=p+1$, we obtain that $H^{1}(X, \underline{H}_{y}^{p}(\Omega_{X/\mathbb{Q}}^{p-1})) = H_{y}^{p+1}(\Omega_{X/\mathbb{Q}}^{p-1})$.

So the map $\pi$(2.2)above is 
\[
\pi : H^{1}(\mathcal{N}_{ Y/ X}) \to H_{y}^{p+1}(\Omega_{X/\mathbb{Q}}^{p-1}).
\]

Composing $\pi$ with the map
\[
i: H_{y}^{p+1}(\Omega_{X/\mathbb{Q}}^{p-1}) \to H^{p+1}(\Omega_{X/\mathbb{Q}}^{p-1}),
\]
we have 
\[
\tilde{\pi} := i \circ \pi : H^{1}(\mathcal{N}_{ Y/ X}) \to H^{p+1}(\Omega_{X/\mathbb{Q}}^{p-1}).
\]

\begin{theorem} \label{theorem: Main}
The map 
\[
\tilde{\pi}: H^{1}(\mathcal{N}_{ Y/ X}) \to H^{p+1}(\Omega_{X/\mathbb{Q}}^{p-1})
\]
is zero.
\end{theorem}

\begin{proof}
This follows from the fact that $H_{y}^{p+1}(\Omega_{X/\mathbb{Q}}^{p-1})=0$(so that the maps $\pi$ and $\tilde{\pi}$
are zero), which can be seen from two different ways. 

Since $\underline{H}_{y}^{p}(\Omega_{X/\mathbb{Q}}^{p-1})$  is a flasque sheaf,  $H_{y}^{p+1}(\Omega_{X/\mathbb{Q}}^{p-1})=H^{1}(X, \underline{H}_{y}^{p}(\Omega_{X/\mathbb{Q}}^{p-1})) =0$.

Alternatively, $H_{y}^{p+1}(\Omega_{X/\mathbb{Q}}^{p-1})=H_{[y]}^{p+1}(\mathrm{Spec}(O_{X,y}), \widetilde{\Omega}_{O_{X,y}/\mathbb{Q}}^{p-1})$, where $[y]$ is the maximal idea of $O_{X,y}$ and 
$\widetilde{\Omega}_{O_{X,y}/\mathbb{Q}}^{p-1}$ is the sheaf associated to $\Omega_{O_{X,y}/\mathbb{Q}}^{p-1}$.
Since $\mathrm{Spec}(O_{X,y})$ has dimension $p$, it is obvious that 
\[
H_{[y]}^{p+1}(\mathrm{Spec}(O_{X,y}), \widetilde{\Omega}_{O_{X,y}/\mathbb{Q}}^{p-1})=0.
\]

\end{proof}

\textbf{Interpretation:}
For $X$ be a smooth projective variety over a field $k$ of characteristic $0$ and 
$Y \subset X$ a locally complete intersection of codimension $p$,  it is known that the obstruction to lifting $Y$ for embedded deformations lies in $ H^{1}(\mathcal{N}_{ Y/ X})$, e.g., see $\mathrm{Theorem \ 6.2}$(page 47) of \cite{Hartshorne2}.

One the other hand, $Y$ defines an element of Chow group $CH^{p}(X)$ and the obstruction to lifting it lies in $H^{p+1}(\Omega_{X/\mathbb{Q}}^{p-1})$, e.g., see \cite{BEK, GGChow}.

From this viewpoint, the map 
\[
\tilde{\pi} :  H^{1}(\mathcal{N}_{ Y/ X}) \to  H^{p+1}(\Omega_{X/\mathbb{Q}}^{p-1})
\]
is a map between obstruction spaces and it is zero, meaning obstructions vanish by viewing $Y$ as a cycle. 

Since $Y \subset X$ is a locally complete intersection, one can locally lift $Y$ and
the obstructions arise when one tries to glue from local to global. Viewed as a cycle,
$Y=\overline{\{y\}}$, there is no necessary to worry about gluing from local to global(see Lemma 3.9 of  \cite{Y-2} for related discussion), since the generic point $y$ lies in any non-empty open subset of $Y$. This is already  reflected in the proof of Theorem 2.2. What we have shown in Theorem 2.2 is $H_{y}^{p+1}(\Omega_{X/\mathbb{Q}}^{p-1})=0$ and we do not know whether $H^{p+1}(\Omega_{X/\mathbb{Q}}^{p-1})=0$ or not.






\textbf{Acknowledgements}
 The author is very grateful to Spencer Bloch for discussions \cite{Bloch1}, and must also record that the idea of eliminating obstructions is due to Mark Green and Phillip Griffiths \cite{GGtangentspace}(page187-190). He thanks all of them for help and encouragement.

\end{document}